\documentclass[graybox]{svmult}

\usepackage{mathptmx}       
\usepackage{helvet}         
\usepackage{courier}        
\usepackage{type1cm}        
\usepackage{makeidx}         
\usepackage{multicol}        
\usepackage[bottom]{footmisc}
\usepackage[normalem]{ulem}  
\usepackage{mathtools}  
\allowdisplaybreaks

\makeindex             

\usepackage{verbatim}
\usepackage[rightcaption]{sidecap}

\usepackage{amsmath}
\usepackage{amsfonts}
\usepackage{enumerate}        

\newcommand{\N}{{\mathbb N}}
\newcommand{\R}{{\mathbb R}}
\newcommand{\C}{{\mathbb C}}

\usepackage{xcolor}

\newcommand{\Cnn}{\C^{n\times n}}
\newcommand{\Cn}{\C^n}

\newcommand{\mHC}{{m_{HC}}}
\newcommand{\Hn}{\mathbb{H}_n}
\newcommand{\PDHn}{\mathbb{H}_n^>}
\newcommand{\tc}{\widetilde{c}}

\newcommand{\SFI}{r}

\newcommand{\mA}{\mathbf{L}}
\newcommand{\mAH}{\mathbf{L}_H}
\newcommand{\mAS}{\mathbf{L}_S}
\newcommand{\mC}{\mathbf{C}}
\newcommand{\mCH}{\mathbf{C}_H}

\newcommand{\mH}{\mathbf{H}}
\newcommand{\mI}{\mathbf{I}}
\newcommand{\mJ}{\mathbf{J}}
\newcommand{\mP}{\mathbf{P}}
\newcommand{\mR}{\mathbf{R}}
\newcommand{\mU}{\mathbf{U}}
\newcommand{\mV}{\mathbf{V}}
\newcommand{\mW}{\mathbf{W}}

\newcommand{\hA}{\widehat{\mA}}

\newcommand{\dd}[1][x]{\,\operatorname{d}\!#1}
\newcommand{\dt}{\,\dd[t]}

\newcommand{\cL}{\mathcal{L}}
\newcommand{\cLAS}{\mathcal{L}_{AS}} 
\newcommand{\cLSD}{\mathcal{L}_{SD}} 
\newcommand{\cLSHscalar}{\mathcal{L}_{IM}} 
\newcommand{\setL}{\mathcal{L}}

\newcommand{\bigO}{\mathcal{O}}

\DeclareMathOperator{\diag}{diag}

\begin{document}

\title{Necessary and sufficient conditions for strong stability of explicit Runge--Kutta methods}
\titlerunning{Strong stability of explicit Runge--Kutta methods}
\author{Franz Achleitner \and Anton Arnold \and Ansgar J\"ungel} 

\institute{Franz Achleitner \at TU Wien, Institute of Analysis and Scientific Computing, Wiedner Hauptstr. 8-10, A-1040 Wien,  Austria, \email{franz.achleitner@tuwien.ac.at}
\and Anton Arnold \at TU Wien, Institute of Analysis and Scientific Computing, Wiedner Hauptstr. 8-10, A-1040 Wien,  Austria, \email{anton.arnold@tuwien.ac.at}
\and Ansgar J\"ungel \at TU Wien, Institute of Analysis and Scientific Computing, Wiedner Hauptstr. 8-10, A-1040 Wien,  Austria, \email{ansgar.juengel@tuwien.ac.at}}
%
%

\maketitle

\abstract{Strong stability is a property of time integration schemes for ODEs that preserve temporal monotonicity of solutions in arbitrary (inner product) norms.
It is proved that explicit Runge--Kutta schemes of order $p\in 4\N$ with $s=p$ stages for linear autonomous ODE systems are not strongly stable, closing an open stability question from [Z.~Sun and C.-W.~Shu, SIAM J. Numer. Anal. 57 (2019), 1158--1182]. Furthermore, for explicit Runge--Kutta methods of order $p\in\N$ and $s>p$ stages, we prove several sufficient as well as necessary conditions for strong stability. These conditions involve both the stability function and the hypocoercivity index of the ODE system matrix. This index is a structural property combining the Hermitian and skew-Hermitian part of the system matrix.
}







\section{Introduction}

Explicit Runge--Kutta methods are standard tools in the numerical solution of ordinary differential equations and semi-discrete approximations of partial differential equations. Of particular interest are strongly stable schemes for which the norm of its numerical solution is nonincreasing in time, guaranteeing that the numerical error in each time step is not amplified during time integration. Surprisingly, the characterization of strongly stable explicit schemes is not complete. For instance, while it is well-known \cite{SunShu19} that explicit Runge--Kutta methods for linear systems of order $p$ are \emph{not} strongly stable if $p\in 4\N_0+1$ or $p\in 4\N_0+2$ and that they are strongly stable for $p\in 4\N_0+3$, the case $p\in 4\N$ is still open. 

In this work, we show that the proof of strong stability can be reduced to testing strong stability with respect to (w.r.t.) the set of asymptotically stable, semi-dissipative matrices (see below for the definitions) and w.r.t.\ the set of purely imaginary scalars. We deduce that explicit Runge--Kutta methods for linear systems with order $p\in 4\N$ and $p=s$ stages are \emph{not} strongly stable, thus filling the gap left open in \cite[Theorem 4.2]{SunShu19}. Furthermore, we prove that strong stability w.r.t.\ asymptotically stable, semi-dissipative matrices holds if the so-called hypocoercivity index of the system matrix is small enough.

We consider linear time-invariant systems of ordinary differential equations,
\begin{equation}\label{ODE:A}
  \frac{\dd[u]}{\dt} = \mA u, \quad t>0, \quad u(0) = u^0\in\C^n,
\end{equation}
with Lyapunov stable matrices~$\mA\in\Cnn$; see \cite[Definition 15.9.1]{Be18}, \cite[p.172]{HadChe08}. In this paper, we use several notions of stability for matrices $\mA\in\Cnn$:
\begin{itemize}
\item $\mA$ is \emph{Lyapunov stable} if all eigenvalues have nonpositive real part and those eigenvalues with vanishing real part are non-defective (i.e., the algebraic and geometric multiplicities coincide) \cite[Definition 15.9.1]{Be18}.
\item $\mA$ is \emph{asymptotically stable} if all eigenvalues have negative real part \cite[Definition 15.9.1]{Be18}.
\item $\mA$ is \emph{dissipative} (resp.~\emph{semi-dissipative}) if its Hermitian part $(\mA+\mA^*)/2$ is negative definite (resp.~negative semi-definite) \cite[Definition 4.1.1]{Be18}.
\end{itemize}
Dissipative matrices are asymptotically stable, and semi-dissipative matrices are Lyapunov stable. Lyapunov stability characterizes all system matrices~$\mA\in\Cnn$ such that $u\equiv 0$ is a Lyapunov stable solution to \eqref{ODE:A}. 

We use the following notation:
The conjugate transpose of a matrix $\mA\in\Cnn$ is denoted by $\mA^*$.
Matrices~$\mA\in\Cnn$ have a unique decomposition $\mA=\mAH+\mAS$ into its Hermitian part~$\mAH:=(\mA+\mA^*)/2$ and skew-Hermitian part~$\mAS:=(\mA-\mA^*)/2$.
We write $\Hn$ for the set of all Hermitian matrices in $\Cnn$. 
Positive definiteness (resp.~semi-definiteness) of Hermitian matrices~$\mP\in\Hn$ is denoted by $\mP>0$ (resp.~$\mP\geq 0$), and
$\PDHn:=\{\mH\in\Hn: \mH>0\}$ is the set of positive definite Hermitian matrices in $\Cnn$. Moreover, $\langle\cdot,\cdot\rangle$ is the standard inner product in $\Cn$ with norm $\|\cdot\|=\sqrt{\langle\cdot,\cdot\rangle}$, and $\|\cdot\|_2$ denotes the spectral norm in $\Cnn$.

It is known that
a matrix~$\mA\in\Cnn$ is Lyapunov stable if and only if there exists a positive definite Hermitian matrix $\mP\in\PDHn$ such that  
\begin{equation} \label{ineq:Lyapunov}
 \mA^* \mP +\mP\mA\leq 0;
\end{equation}  
see, e.g., \cite[Corollary to Theorem 5]{Str75}, \cite[Theorem 3.18]{HadChe08}. 
Under this condition, the solution $u(t)$ of~\eqref{ODE:A} is nonincreasing in the norm $\|\cdot\|_\mP :=\sqrt{\langle\cdot,\mP\cdot\rangle}$:
\begin{equation} \label{ineq.monotonicity}
  \frac{\dd[]}{\dt}\|u(t)\|_\mP^2
= \frac{\dd[]}{\dt}\langle u,\mP u\rangle
= \langle \mA u,\mP u\rangle +\langle u,\mP\mA u\rangle 
= \langle u,(\mA^*\mP+\mP\mA)u\rangle 
\le 0.
\end{equation}
Often, the aim of numerical schemes is to reproduce this property on the discrete level, i.e.\ to show that a certain norm of an approximation $u^k$ of $u(k\tau)$, where $\tau>0$ is the uniform time step, is nonincreasing in $k$. For this, we consider the explicit Runge--Kutta methods
\begin{equation}\label{1.RK}
  u^k 
= u^{k-1} + \tau\sum_{i=1}^s b_i K_i^k, \quad
  K_i^k
= \mA\bigg(u^{k-1} + \tau\sum_{j=1}^{i-1}a_{ij}K_j^k\bigg), \quad
	i=1,\ldots,s,
\end{equation}
where $b_i\in\C$ are the weights, $a_{ij}\in\C$ are the coefficients of the Runge--Kutta matrix, and $s\in\N$ is the number of stages.
Stability of explicit Runge--Kutta schemes can be expected only under a restriction on the time step. 
In this article, we give necessary and sufficient conditions such that explicit Runge--Kutta methods preserve the monotonicity property~\eqref{ineq.monotonicity} in the following sense.

\begin{definition}[Strongly stable]\label{def.stab}
\begin{enumerate}
    \item[(a)]
The Runge--Kutta scheme \eqref{1.RK} is {\em strongly stable} if for all matrix dimensions $n\in\N$, for all Lyapunov stable matrices $\mA\in\Cnn$, and for all $\mP\in\PDHn$ such that~\eqref{ineq:Lyapunov} holds, the numerical solution to \eqref{ODE:A} satisfies $\|u^1\|_\mP\le\|u^0\|_\mP$ for all initial data $u^0\in\C^n$ and sufficiently small time steps.

    \item[(b)]
For practical reasons, we call the strongly stable Runge--Kutta scheme \eqref{1.RK} 
\emph{strongly stable w.r.t. the set~$\setL$ of Lyapunov stable matrices} of any matrix dimension $n\in\N$,
\begin{equation*} 
 \setL :=\bigcup_{n\in\N}\{ \mA\in\Cnn: \exists \mP\in\PDHn\text{ such that~\eqref{ineq:Lyapunov} holds}\}.
\end{equation*}

    \item[(c)]
Similarly, the Runge--Kutta scheme \eqref{1.RK} is called \emph{strongly stable w.r.t.~a subset $\setL_0$ of~$\setL$} if the above property holds for all 
$\mA\in\setL_0$. 
\end{enumerate}
\end{definition}
\begin{remark}\label{rem:stab}
\begin{enumerate}[(a)]
\item
For explicit Runge--Kutta methods, strong stability w.r.t.~$\cL$ is equivalent to strong stability w.r.t. \emph{semi-dissipative} matrices~$\mA\in\cLSD$ with
$$ 
  \cLSD := \bigcup_{n\in\N}\{\mA\in\Cnn: \mA+\mA^*\leq 0\} 
  \subset \setL; 
$$
see the first step in the proof of Theorem~\ref{thm.stab} below.
\item
Strong stability of explicit Runge--Kutta schemes w.r.t.~$\cL$ 
(as studied here and in~\cite[\S4]{Tad02}, \cite{SunShu19}) is ``more restrictive'' than strong stability w.r.t.~(uniformly) coercive system matrices $\mA$ as studied in~\cite{LeTa98}, \cite[\S3]{Tad02}.
\end{enumerate}
\end{remark}

The goal of this paper is twofold: 
On the one hand, we give novel necessary conditions for strong stability of explicit Runge--Kutta schemes.
On the other hand, we derive sufficient conditions for strong stability in a weaker sense:
In contrast to Definition~\ref{def.stab}, we then require the monotonicity property only in {\em some} weighted norm $\|\cdot\|_\mP$ with~$\mP\in\PDHn$.

\begin{definition}[Strong stability in weak form]
\begin{enumerate}
    \item[(a)]
The Runge--Kutta scheme~\eqref{1.RK} is {\em strongly stable in weak form} if for all $n\in\N$ and for all Lyapunov stable matrices~$\mA\in\Cnn$, there exists a matrix~$\mP\in\PDHn$ satisfying \eqref{ineq:Lyapunov}, 
such that the numerical solution~\eqref{1.RK} to~\eqref{ODE:A} satisfies $\|u^1\|_\mP \le\|u^0\|_\mP$ for all initial data $u^0\in\C^n$ and sufficiently small time steps. 
    \item[(b)]
Similarly, the Runge--Kutta scheme \eqref{1.RK} is called \emph{strongly stable in weak form w.r.t.~a subset $\setL_0$ of~$\setL$} if the above property holds for all $\mA\in\setL_0$.
\end{enumerate}
\end{definition}

This notion of \emph{strong stability in weak sense} is inspired by an example in \cite[\S3.5]{LeTa98}; see Example \ref{ex:Levy.Tadmor}.

Most of the practical interest in strongly stable  Runge--Kutta methods has been for explicit schemes, e.g. in the integration of hyperbolic conservation laws \cite{Shu88,ShuOsh88}.
For implicit or implicit-explicit strongly stable Runge--Kutta methods, we refer to, e.g.,  \cite{CGGS17,FeSp04,GST01,Hig04,SunWeiWu22}. 
Furthermore, nonlinear problems are considered in, e.g., \cite{Ran21,RaKe20}. 
Note that a Runge--Kutta method of linear order $p$ (i.e., the order for linear systems) may possess a lower order when applied to nonlinear systems. 

Strong stability of the (implicit) midpoint rule can be seen as a corollary to Remark~\ref{rem:stab}(a) and \cite[Lemma 47(ii)]{AAM23}.

\bigskip
Let us summarize the main results of this paper:
\begin{itemize}
    \item 
An explicit Runge-Kutta scheme is strongly stable if and only if it is strongly stable w.r.t.~the set~$\cLAS$ of asymptotically stable, semi-dissipative matrices and w.r.t.~the set~$\cLSHscalar$ of purely imaginary scalars, see Theorem~\ref{thm.stab}.
    \item 
Explicit Runge--Kutta schemes 
of order $p\in 4\N$ and $s=p$ stages are {\em not} strongly stable. 
While these Runge--Kutta schemes are strongly stable w.r.t.~$\cLSHscalar$, they fail to be strongly stable w.r.t.~$\cLAS$: see Corollary~\ref{cor:ERK_p4N}. 
For this analysis, we derive novel necessary conditions for strong stability of explicit Runge--Kutta schemes using the precise short-time asymptotics for the spectral norm of the matrix exponential as derived in~\cite[Theorem 2.7]{AAC22}.
    \item 
Concerning strong stability w.r.t.~$\cLAS$, we show a remarkable connection between the order $p\in\N$ of the Runge--Kutta scheme and the hypocoercivity (HC) index of asymptotically stable, semi-dissipative matrices; see Theorem \ref{thm.main}. Roughly speaking, the HC-index of a matrix~$\mA$ describes the structural complexity of the intertwining of the Hermitian part~$\mAH$ and skew-Hermitian part~$\mAS$; see Definition \ref{def:HCI}.  
    \item 
Finally, we show that each explicit Runge--Kutta method of order $p\in\N$ is strongly stable in weak form if and only if it is 
locally stable on the imaginary axis as defined in~\cite[Definition 2.1]{KrSc92}; see Theorem~\ref{thm.StrongStabilityInWeakForm}.    
\end{itemize}


\section{Main results}
\label{sec:main_results}

Before stating our main theorems, we recall that the Runge--Kutta scheme~\eqref{1.RK} for ODE systems~\eqref{ODE:A} can be written as $u^k=R(\tau\mA)u^{k-1}$, where $R(z)$ is the stability function \cite[Definition 2.1]{HaWa96}. 
For linear time-invariant systems, $s=p$ stages are sufficient to obtain a scheme of order $p$~\cite[\S~IV.2]{HaWa96}.
Generally, if an explicit Runge--Kutta method is of order $p$, its stability function can be written as
\begin{equation}\label{1.R}
 R(z) 
=\sum_{j=0}^p\frac{z^j}{j!} + \sum_{j=p+1}^s c_j\frac{z^j}{j!}, 
\quad z\in\C,\ c_{p+1}\neq 1 ,
\end{equation}
with some constants $c_{p+1},\ldots,c_s\in\R$ \cite[Theorem IV.2.2]{HaWa96} and $c_j=0$ for $j\geq s+1$.
For an explicit Runge--Kutta scheme with $s=p$, this expression reduces to the truncated exponential $R(z)=\sum_{j=0}^p z^j/j!$; see, e.g., \cite[(2.12)]{HaWa96}.

Our first main result shows that, to establish strong stability for an explicit Runge--Kutta scheme, it is (necessary and) sufficient to prove strong stability w.r.t.~two distinct subsets of Lyapunov stable matrices.

\begin{theorem}\label{thm.stab}
Consider an explicit Runge--Kutta method with stability function~\eqref{1.RK}.
Then the Runge--Kutta scheme is strongly stable if and only if it is strongly stable w.r.t. 
\begin{enumerate}[(a)]
 \item 
the set~$\cLAS$ of asymptotically stable, semi-dissipative matrices, and 
 \item 
the set~$\cLSHscalar$ of purely imaginary scalars $\mA\in \I\R\subset\C$. 
\end{enumerate}
\end{theorem}
This means that the Runge--Kutta scheme is strongly stable if and only if for all linear systems~\eqref{ODE:A} with~$\mA$ either in~$\cLAS$ or~$\cLSHscalar$, for sufficiently small time steps, the numerical solution to~\eqref{ODE:A} satisfies 
$\|u^1\|\leq \|u^0\|$ for all initial data $u^0\in\Cn$. 

The (strong) stability of Runge--Kutta schemes for scalar differential equations is well studied; see, e.g., \cite{KrSc92}.
In particular, strong stability w.r.t.~$\cLSHscalar$ is equivalent to \emph{local stability on the imaginary axis} (i.e., there exists $Z>0$ such that $|R(z)|\le 1$ for all $z\in\I\R$ with $|z|\le Z$). For example, the Runge--Kutta method with $p=s=4$ is locally stable on the imaginary axis. 

Therefore, we shall focus here on strong stability w.r.t.~asymptotically stable, semi-dissipative matrices~$\mA\in\cLAS$.
We shall show that a Runge--Kutta method with $p=s=4$ fails to be strongly stable for certain asymptotically stable, semi-dissipative matrices $\mA\in\cLAS$. We also discuss the (counter-)examples given in~\cite[\S3.5]{LeTa98} and~\cite[Proposition 1.1]{SunShu17}.

Our second main result is concerned with strong stability w.r.t. asymptotically stable, semi-dissipative matrices~$\cLAS$.
We show a remarkable connection between the order $p\in\N$ of the Runge--Kutta scheme and the hypocoercivity index (HC-index) of asymptotically stable, semi-dissipative matrices $\mA\in\cLAS$, which is defined as follows; see \cite[Definition 3.1]{AAC22}, \cite[Definition 3]{AAM21}.
\begin{definition} \label{def:HCI}
Let $\mA\in\Cnn$ be semi-dis\-si\-pa\-tive.
The~\emph{hy\-po\-co\-er\-ci\-vi\-ty index (HC-index)~$m_{HC}=m_{HC}(\mA)$ of the (semi-dissipative) matrix~$\mA$} is defined as the smallest integer~$m\in\N_0$ (if it exists) such that
\begin{equation}\label{mHC-ineq}
  T_m :=
  \sum_{j=0}^m \mAS^j \mAH (\mAS^*)^j < 0.
\end{equation}
\end{definition}

\begin{remark}
Originally, the HC-index has been defined for accretive matrices $\mC\in\Cnn$ (i.e., matrices with positive semi-definite Hermitian part~$\mCH\geq 0$) in~\cite{AAC18}. 
For practical reasons, we present the hypocoercivity theory here for semi-dissipative matrices (using the fact that for a semi-dissipative matrix~$\mA\in\Cnn$, the matrix $\mC:=-\mA$ is accretive).
\end{remark}

By definition, a semi-dissipative matrix~$\mA$ is dissipative if and only if $m_{HC}(\mA)=0$.
We recall from \cite[Lemma 2.4]{AAC18} that every matrix $\mA\in\cLAS$ has a finite HC-index. For the next theorem, we introduce the following subsets of $\cLAS$:
$$
  \cLAS^{m}:=\{\mA\in\cLAS\: :\ \mHC(\mA)\le m \}, \quad m\in\N_0.
$$

\begin{theorem}\label{thm.main}
Consider explicit Runge--Kutta schemes of order $p\in \N$ for \eqref{ODE:A} with stability function~\eqref{1.R}.
Then the following results for strong stability w.r.t.~asymptotically stable, semi-dissipative matrices~$\mA\in\cLAS$ hold:
\begin{enumerate}[(a)]
 \item
All explicit Runge--Kutta schemes of order $p\in \N$ are strongly stable w.r.t.~$\cLAS^{m}$, if $m\in\N_0$ satisfies $2m+1\leq p$. 
 \item
An explicit Runge--Kutta scheme of order~$p\in\N$ is \emph{not} strongly stable w.r.t.~$\cLAS$ if the coefficient $c_{p+1}$ in \eqref{1.R} satisfies
\begin{subequations} \label{c}
\begin{align} 
  (-1)^{(p+1)/2}(1-c_{p+1}) < 0 & \quad\mbox{ for $p$  odd;} \label{c.odd} \\
\text{or}\qquad  
  1 +(-1)^{p/2} (c_{p+1} -1)\tbinom{p}{p/2} < 0 & \quad \mbox{ for $p$ even.} 
	\label{c.even}
\end{align}
\end{subequations}
\item
In case~\eqref{c} holds and the matrix $\mA\in\cLAS$ satisfies $2\mHC(\mA)+1>p$, the considered Runge--Kutta scheme is \emph{not} strongly stable w.r.t.~the singleton $\{\mA\}$. This means that it is strongly stable w.r.t.~$\cLAS^m$ but \emph{not} strongly stable w.r.t.~$\cLAS\setminus\cLAS^{m}$ for any $m\in\N$ with $2m+1\le p$.
\end{enumerate}
\end{theorem}

The conditions \eqref{c} on the coefficients of the stability functions directly yield the following result (which is also included in Table \ref{table:ERK} and illustrated in the examples in Section \ref{sec.stab}).

\begin{corollary}\label{cor:ERK_p4N}
Explicit Runge--Kutta schemes of order $p\in 4\N$ or~$p\in 4\N_0+1$ with $s=p$ stages (and hence with $c_{p+1}=0$) are \emph{not} strongly stable. 
\end{corollary}

Next, we shall combine two sufficient criteria for explicit Runge--Kutta schemes to \emph{fail} strong stability: the criterion to fail strong stability w.r.t.~$\cLAS$ (given in Theorem~\ref{thm.main}(b)), and w.r.t.~$\cLSHscalar$ (given in Theorem~\ref{thm.LS} ---  and taken from the literature: \cite[Theorem 3.1]{KrSc92}).

\begin{theorem} \label{thm.ERK.notStronglyStable}
Explicit Runge--Kutta schemes of order $p\in \N$ for \eqref{ODE:A} 
are \emph{not} strongly stable if
\begin{subequations} \label{thm.ERK.notStronglyStable.c}
\begin{align} 
  (-1)^{(p+1)/2}(1-c_{p+1})<0, &\quad\mbox{for $p$ odd;} 
  \label{podd} \\
  \begin{rcases}
  (-1)^{p/2}\big(c_{p+2} - (p+2)c_{p+1} + (p+1)\big) < 0 \\
  \mbox{or }\quad 1 +(-1)^{p/2} (c_{p+1} -1)\tbinom{p}{p/2} < 0
  \end{rcases} &\quad\mbox{for $p$ even.}
  \label{thm.ERK.notStronglyStable.even}
\end{align}
\end{subequations}  
\end{theorem}

\begin{remark} 
The negation of Theorem~\ref{thm.ERK.notStronglyStable} represents necessary conditions for strong stability. As such it 
complements the sufficient condition for strongly stable, explicit Runge--Kutta schemes of even order $p\in 2\N$ given in~\cite[Theorem 4.5]{SunShu19}. More precisely, the negation of the first condition of \eqref{thm.ERK.notStronglyStable.even} coincides with the first condition in~\cite[Theorem 4.5]{SunShu19}, except of the (non-)strict relation sign. The negation of the second condition from \eqref{thm.ERK.notStronglyStable.even} and its analog from~\cite[Theorem 4.5]{SunShu19} both give a bound on $c_{p+1}$, but they leave a gap between the sufficient and necessary conditions. We notice that condition \eqref{podd} is sufficient and necessary in \cite{SunShu19}.
In Table~\ref{table:ERK}, we summarize the results for explicit Runge--Kutta schemes of order $p\in\N$ and $s=p$ stages.
\end{remark}

\begin{SCtable}[2.8]
\begin{tabular}{llll}
$p$ & $\cLAS$ & $\cLSHscalar$ & $\cL$ \\[1mm]
\hline\\[-3mm]
$4\N_0+1$ & No & No & No \\
$4\N_0+2$ & ? & No & No \\
$4\N_0+3$ & Yes & Yes & Yes \\
$4\N$ & \textbf{No} & Yes & \textbf{No} \\
\hline\\
\end{tabular}
\caption{shows whether or not explicit Runge--Kutta methods of order $p\in\N$ with $s=p$ stages are strongly stable w.r.t. asymptotically stable, semi-dissipative matrices~$\cLAS$, purely imaginary scalars~$\cLSHscalar$, and Lyapunov stable matrices~$\cL$, respectively; new results are in bold face. 
Theorem~\ref{thm.ERK.notStronglyStable} shows that explicit Runge--Kutta schemes for linear systems of order $p\in 4\N$ and $s=p$ stages are \textbf{\em not} strongly stable. This closes the open case in~\cite[Section 4.2]{SunShu19}.}
\label{table:ERK}
\end{SCtable}

Our third main result is concerned with sufficient conditions for strong stability in weak form. 

\begin{theorem} \label{thm.StrongStabilityInWeakForm}
An explicit Runge--Kutta scheme of order $p\in\N$ is strongly stable in weak form if and only if it is strongly stable w.r.t.~$\cLSHscalar$ (i.e., if and only if it is locally stable on the imaginary axis).
\end{theorem}

Due to Theorem~\ref{thm.stab}, an explicit Runge--Kutta method is strongly stable if and only if it is strongly stable w.r.t. to the sets~$\cLAS$ and~$\cLSHscalar$.
For asymptotically stable, semi-dissipative matrices~$\mA\in\cLAS$, using a suitable, modified inner product norm ensures  monotonicity/strong stability.

\bigskip
This paper is organized as follows.
We prove Theorem \ref{thm.stab} in Section \ref{sec.stab}, while Section \ref{sec.main} is concerned with the proof of Theorem \ref{thm.main}. Finally, we present the proof of Theorem \ref{thm.StrongStabilityInWeakForm} in Section \ref{sec:StrongStability.of.ERK:weak.form}.


\section{Proof of Theorem \ref{thm.stab}}\label{sec.stab}

The proof is based on a reduction of the matrix $\mA$ to block diagonal form. First, we recall some preparatory results and then detail the reduction strategy.

\medskip\noindent\textbf{Preparations.}
It is well-known that a similarity transformation of a Lyapunov stable matrix yields a semi-dissipative transformed matrix.

\begin{lemma}\label{lem.sim}
Let $\mA\in\Cnn$ be Lyapunov stable. Then there exists $\mP\in\PDHn$ such that $\widehat{\mA}:=\mP^{1/2}\mA\mP^{-1/2}$ is semi-dissipative and $\widehat{u}:=\mP^{1/2}u$ transforms \eqref{ODE:A} into the semi-dissipative ODE system $\mathrm{d}\widehat{u}/\mathrm{d}t=\widehat{\mA}\widehat{u}$. Here, $\mP^{-1/2}:=(\mP^{-1})^{1/2}$.
\end{lemma}

\begin{proof}
Since $\mA$ is Lyapunov stable, there exists  $\mP\in\PDHn$ such that $\mA^*\mP+\mP\mA\le 0$. A congruence transformation with the Hermitian matrix $\mP^{-1/2}$ yields
$$
  0\ge \mP^{-1/2}(\mA^*\mP +\mP\mA)\mP^{-1/2}
   = \mP^{-1/2} \mA^* \mP^{1/2} +\mP^{1/2} \mA \mP^{-1/2}
   = 2(\mP^{1/2} \mA \mP^{-1/2})_H,
$$
proving that $\widehat{\mA}$ is semi-dissipative. Finally, we multiply \eqref{ODE:A} from the left by $\mP^{1/2}$ to find that
$\mathrm{d}\widehat{u}/\mathrm{d}t = \mP^{1/2}\mA u = (\mP^{1/2}\mA\mP^{-1/2})\widehat{u}= \widehat{\mA}\widehat{u}$.
\hfill$\Box$
\end{proof}

We can characterize semi-dissipative matrices that are {\em not} asymptotically stable; see \cite[Lemma 3.1]{MehMS16}, \cite[Lemma 2.4 with Prop.~1 (B2), (B4)]{AAC18}.

\begin{proposition} \label{prop:border}
Let $\mA\in\Cnn$ be semi-dissipative. Then $\mA$ has an eigenvalue on the imaginary axis (and is hence \emph{not} asymptotically stable) if and only if $\mAH v =0$ for some eigenvector~$v$ of~$\mAS$.
\end{proposition}

Note that, due to the assumption on $\mA$, purely imaginary eigenvalues of semi-dissipative matrices are necessarily semi-simple, i.e., their algebraic and geometric multiplicities coincide; see \cite{MehMW18,MehMW20}.

\emph{Unitary congruence transformations} $\mA\mapsto \mU\mA \mU^*$ with unitary matrices~$\mU$, preserve (asymptotic) stability, semi-dissipativity, and the HC-index (cf.\ \eqref{mHC-ineq}).
For semi-dissipative matrices~$\mA\in\Cnn$, the pair $(\mJ,\mR):=(\mAS,\mAH)$ can be transformed into a (so-called) \emph{staircase form} by using unitary congruence transformations \cite[Lemma 57]{AAM23}.

\begin{lemma}[Staircase form for $(\mJ,\mR)$] \label{lem:SF}
Let $\mJ\in\Cnn$ be a skew-Hermitian matrix and $\mR\in\Cnn$ be a nonzero Hermitian matrix.
Then there exists a unitary matrix $\mV\in\Cnn$ such that~$\mV \mJ \mV^*$ and~$\mV \mR \mV^*$ are block tridiagonal matrices of the form
{\small
\setlength\arraycolsep{4pt}
\begin{subequations}\label{matrices:staircase:J_R}
\begin{eqnarray} \label{matrices:staircase:J}
\mV \mJ \mV^*
\!\!\!&=&\!\!\!\!\! \begin{array}{l}
\left[ \begin{array}{ccccccc|c}
 \mJ_{1,1} & -\mJ_{2,1}^* & & & \cdots & & 0 & 0\\
 \mJ_{2,1} & \mJ_{2,2} & -\mJ_{3,2}^* & & & & &\\
  & \ddots & \ddots & \ddots & & & \vdots & \\
  & & \mJ_{k,k-1} & \mJ_{k,k} & -\mJ_{k+1,k}^* & & & \vdots \\
 \vdots & & & \ddots & \ddots & \ddots & & \\
  & & & & \mJ_{\SFI-2,\SFI-3} & \mJ_{\SFI-2,\SFI-2} & -\mJ_{\SFI-1,\SFI-2}^* & \\
 0 & & \cdots & & & \mJ_{\SFI-1,\SFI-2} & \mJ_{\SFI-1,\SFI-1} & 0\\ \hline
 0 & & & \cdots & & & 0 & \mJ_{\SFI,\SFI}
\end{array}\right]
 \begin{array}{c}
  n_1\\ n_2\\[2pt] \vdots\\[4pt] n_k\\[2pt] \vdots \\[2pt] n_{\SFI-2} \\n_{\SFI-1}\\[1pt] n_\SFI
 \end{array} \\
 \quad\hspace{5pt} n_1 \hspace{130pt} n_{\SFI-2} \hspace{22pt} n_{\SFI-1}  \hspace{18pt} n_\SFI,
\end{array} \\
&&\nonumber\\
\mV \mR \mV^*
\!\!\!&=&\!\!\!\!\! \begin{array}{l}
\left[ \begin{array}{cccc|c}
 \mR_1 & 0 & \cdots & 0 & 0 \\
   0 & 0 & & \vdots & \vdots \\
 \vdots & & \ddots & \vdots & \vdots \\
 0 & \cdots & \cdots & 0 & 0 \\
\hline
 0 & \cdots & \cdots & 0 & 0 
\end{array}\right]
 \begin{array}{c}
  n_1\\[4pt] n_2\\ \vdots\\ n_{\SFI-1}\\ n_\SFI
 \end{array} \\
 \quad\hspace{1pt} n_1 \hspace{8pt} n_2 \hspace{4pt} \cdots \hspace{4pt} n_{\SFI-1} \hspace{6pt} n_\SFI,
\end{array}
\label{matrices:staircase:J-R}
\end{eqnarray}
\end{subequations}}
where $n_1 \geq n_2 \geq \cdots \geq n_{\SFI-1} >0$, $n_\SFI \geq 0$, and $\mR_1\in\C^{n_1,n_1}$ is nonsingular.

If $\mR$ is nonsingular, then~$\SFI=2$ and~$n_2=0$.
For example, $\mV=\mI$, $\mJ_{1,1}=\mJ$, and~$\mR_1=\mR$ is an admissible choice.

If $\mR$ is singular, then $\SFI\geq 3$ and the matrices~$\mJ_{i,i-1}$, $i=2,\ldots,\SFI-1$, in the subdiagonal have full row rank and are of the form
\begin{equation*}
\mJ_{i,i-1} = \begin{bmatrix} \Sigma_{i,i-1} & 0 \end{bmatrix}, \quad
i =2,\ldots, \SFI-1,
\end{equation*}
with nonsingular matrices~$\Sigma_{i,i-1}\in\C^{n_i\times n_i}$, and moreover $\Sigma_{\SFI-1,\SFI-2}$ is a real-valued diagonal matrix.
\end{lemma}

For semi-dissipative matrices $\mA\in\Cnn$, the staircase form of $(\mJ,\mR)=(\mAS,\mAH)$ can be used to derive a unitary congruence transformation to a block diagonal form.

\begin{lemma} \label{lem:blockDiagonalForm}
Let $\mA\in\Cnn$ be semi-dissipative.
Then there exists a unitary matrix $\mV\in\Cnn$ such that~$\mV \mA \mV^*$ is a block diagonal matrix of the form
\begin{equation} \label{A.staircaseForm}
 \mV\ \mA\ \mV^*
= 
\left[ \begin{array}{c|c}
 \mA_1 & 0\\
\hline
 0 & \mA_2
\end{array}\right],
\end{equation}
where $\mA_1\in\C^{\tilde n_1\times \tilde n_1}$, $\mA_2\in\C^{\tilde n_2\times \tilde n_2}$ with $\tilde n_1,\tilde n_2\in\N_0$ such that $\tilde n_2 =n-\tilde n_1$.
Moreover, the semi-dissipative matrix $\mA_1\in\C^{\tilde n_1\times \tilde n_1}$ is asymptotically stable, and the semi-dissipative matrix $\mA_2\in\C^{\tilde n_2\times \tilde n_2}$ is skew-Hermitian.
\end{lemma}
In the block-diagonal form~\eqref{A.staircaseForm}, not both matrices $\mA_1,\mA_2$ have to be present. Indeed,
if $\mA$ is skew-Hermitian then $\tilde n_1=0$, $\tilde n_2=n$ and $\mA_2 =\mA$, and if $\mA$ is asymptotically stable then $\tilde n_1=n$, $\tilde n_2=0$ and $\mA_1 =\mA$.
\begin{proof}[of Lemma \ref{lem:blockDiagonalForm}]
If the Hermitian part~$\mAH$ of~$\mA$ is the null matrix then we set $\mV=\mI$, $\mA_2=\mA=\mAS$, $\tilde n_1=0$, and $\tilde n_2=n$.

If the Hermitian part $\mAH$ is not the null matrix, we deduce from Lemma~\ref{lem:SF} 
for the pair $(\mJ,\mR):=(\mAS,\mAH)$ the existence of a unitary matrix $\mV\in\Cnn$ 
such that~$\mV\mA\mV^* = \mV\mJ\mV^* + \mV\mR\mV^*$ given in~\eqref{matrices:staircase:J_R} holds with $\mR_1<0$ (although the matrix $\mJ_{\SFI,\SFI}$ may not be present).
If the matrix $\mJ_{\SFI,\SFI}$ is present, then $\mA_2=\mJ_{\SFI,\SFI}$, $\tilde n_2=n_\SFI$, $\tilde n_1=n-n_\SFI>0$, otherwise~$\tilde n_2=0$ and $\tilde n_1=n$.
If the matrix $\mA_2$ is present, it is skew-Hermitian, since it is a (non-leading) principal minor of the skew-Hermitian matrix $\mV\mAS \mV^*$. 

Because of $\mAH\ne0$, the matrix $\mA_1\ne 0$ is the leading principal minor of $\mV\mA \mV^*$ given as 
\begin{align*}
\mA_1
:= 
\begin{array}{l}
\left[ \begin{array}{ccccccc}
 \mJ_{1,1}+\mR_1 & -\mJ_{2,1}^* & & & \cdots & & 0 \\
 \mJ_{2,1} & \mJ_{2,2} & -\mJ_{3,2}^* & & & & \\
  & \ddots & \ddots & \ddots & & & \vdots \\
  & & \mJ_{k,k-1} & \mJ_{k,k} & -\mJ_{k+1,k}^* & & \\
 \vdots & & & \ddots & \ddots & \ddots & \\
  & & & & \mJ_{\SFI-2,\SFI-3} & \mJ_{\SFI-2,\SFI-2} & -\mJ_{\SFI-1,\SFI-2}^* \\
 0 & & \cdots & & & \mJ_{\SFI-1,\SFI-2} & \mJ_{\SFI-1,\SFI-1} 
\end{array}\right]
 \begin{array}{c}
  n_1\\ n_2\\[2pt] \vdots\\[2pt] n_k\\[1pt] \vdots \\ n_{\SFI-2} \\n_{\SFI-1}
 \end{array}, \\
 \quad\hspace{5pt} n_1 \hspace{140pt} n_{\SFI-2} \hspace{25pt} n_{\SFI-1}  
\end{array} 
\end{align*}
and $\mA_1\in \C^{\tilde n_1\times \tilde n_1}$ is semi-dissipative. 

Finally, we apply Proposition~\ref{prop:border} in $\C^{\tilde n_1}$ to deduce that $\mA_1$ is asymptotically stable. If this were not true, there would exist a vector $\tilde v_1=(0,v_2,...,v_{r-1})^T$ in the kernel of $(\mA_1)_H=\diag(\mR_1,0,...,0)$ (see \eqref{matrices:staircase:J-R}) that is an eigenvector of $(\mA_1)_S$. The latter has the block tridiagonal shape in \eqref{matrices:staircase:J}. Since $\mJ^*_{i,i-1}$, $i=2,...,r-1$, has full column rank, we obtain iteratively $\tilde v_1=0$, which cannot be an eigenvector of $(\mA_1)_S$.
\hfill $\Box$
\end{proof}
\begin{remark}\label{stair-HC}
A semi-dissipative matrix $\mA=\mAS+\mAH$ is asymptotically stable if and only if $n_\SFI=0$ holds
in the staircase form~\eqref{matrices:staircase:J-R} of $(\mJ,\mR)=(\mAS,\mAH)$; see the final step in the proof of Lemma~\ref{lem:blockDiagonalForm}. 
If this is the case, the HC-index satisfies $\mHC(\mA)=\SFI-2$; see~\cite[Lemma 4]{AAM21}.
\end{remark}


\medskip\noindent\textbf{Reduction steps.}
We prove the structural result in Theorem~\ref{thm.stab} by performing several reduction steps.

\emph{Step 1. Reduction to semi-dissipative matrices.}
First, we show that strong stability of explicit Runge--Kutta schemes w.r.t. Lyapunov stable matrices $\mA\in\cL$ is equivalent to strong stability w.r.t. semi-dissipative matrices $\hA\in\cLSD$ (as already stated in  Remark~\ref{rem:stab}). To this end,
we consider a Lyapunov stable matrix $\mA\in\Cnn$ and $\mP\in\PDHn$ satisfying~\eqref{ineq:Lyapunov}. By Lemma \ref{lem.sim},
$\hA :=\mP^{1/2}\mA\mP^{-1/2}$ is semi-dissipative. We define
$\widehat{u}^0:=\mP^{1/2}u^0$, $u^1:=R(\tau\mA)u^0$, and $\widehat{u}^1:=R(\tau\hA)\widehat{u}^0$. Then
$\|u^0\|_\mP =\|\mP^{1/2} u^0\| =\|\widehat{u}^0\|$ and 
\begin{align*}
  \|u^1\|_\mP &=\|R(\tau \mA) u^0\|_\mP =\|\mP^{1/2} R(\tau \mA) u^0\| \\
  &=\|R(\tau \mP^{1/2}\mA\mP^{-1/2}) \mP^{1/2} u^0\| 
  = \|R(\tau \hA) \widehat{u}^0\| = \|\widehat{u}^1\| .
\end{align*}
Hence, $\|u^1\|_\mP\leq \|u^0\|_\mP$ if and only if $\|\widehat{u}^1\| \leq\|\widehat{u}^0\|$.

\emph{Step 2. Reduction to semi-dissipative matrices in block diagonal form.}
Considering the unitary congruence transformation from  Lemma~\ref{lem:blockDiagonalForm}, we note that $\|\mV u^j\|=\|u^j\|$, with the unitary matrix $\mV$ from \eqref{A.staircaseForm}.
Therefore, an explicit Runge--Kutta scheme is strongly stable if and only if it is strongly stable w.r.t. semi-dissipative matrices~$\mA$ of the form~\eqref{A.staircaseForm}.

\emph{Step 3. Reduction to two distinct subsets of semi-dissipative matrices.}
Due to the block diagonal structure of~$\mA$ in~\eqref{A.staircaseForm}, an explicit Runge--Kutta scheme is strongly stable if and only if it is strongly stable w.r.t.~two distinct subsets of semi-dissipative matrices, namely
\begin{enumerate}[(a)]
 \item 
asymptotically stable, semi-dissipative matrices $\mA\in\Cnn$, and
 \item \label{case:SH}
skew-Hermitian matrices $\mA\in\Cnn$ (hence satisfying $\mA+\mA^*=0$).
\end{enumerate}

\emph{Step 4. Reduction of skew-Hermitian matrices to (purely imaginary) scalar problems.}
Skew-Hermitian matrices $\mA\in\Cnn$ are unitarily congruent to a diagonal matrix with purely imaginary eigenvalues, i.e., there exists a unitary matrix $\mU\in\Cnn$ such that $\mU\mA \mU^* =\Lambda$ with $\Lambda=\diag(\lambda_1^\mA,\ldots,\lambda_n^\mA)$, where $ \lambda_j^\mA\in \I\R$.
Therefore, in case~\eqref{case:SH}, the analysis can be reduced to scalar ODEs of the form 
\begin{equation*} 
 \frac{\mathrm{d}u}{\mathrm{d}t} =\lambda u, 
\quad t\geq 0; 
\quad \lambda\in \I\R.
\end{equation*}
This finishes the proof of Theorem~\ref{thm.stab}. 
\hfill $\Box$

\medskip
We already mentioned that the (strong) stability of Runge--Kutta schemes for scalar ODEs is well studied; see, e.g., \cite{KrSc92}.
By definition, strong stability w.r.t.~$\cLSHscalar$ is equivalent to \emph{local stability on the imaginary axis} as defined in~\cite[Definition 2.1]{KrSc92} (and recalled in \S\ref{sec:main_results}).
Thus, we deduce from \cite[Theorem 3.1]{KrSc92} the following statement.

\begin{theorem}\label{thm.LS}
The explicit Runge--Kutta method ~\eqref{1.RK} of order $p\in\N$ is strongly stable w.r.t.~$\cLSHscalar$ if
\begin{align*} 
  \gamma_{p+1} :=
  (-1)^{(p+1)/2}(1-c_{p+1}) > 0 &\quad \mbox{ for $p$ odd,} \\
  \delta_{p+1}:=(-1)^{p/2}\big(c_{p+2} - (p+2)c_{p+1} + (p+1)\big) > 0
  &\quad \mbox{ for $p$ even.}
\end{align*} 
If $\gamma_{p+1}<0$ or $\delta_{p+1}<0$, the Runge--Kutta method is not strongly stable. 
\end{theorem}

Notice that the case $\delta_{p+1}=0$ is left open in~\cite[Theorem 3.1]{KrSc92}. 
Complementary, $\gamma_{p+1}=0$ implies $c_{p+1}=1$, this would contradict that the Runge--Kutta method has order $p$, see~\eqref{1.R}.


\medskip\noindent\textbf{Examples to illustrate Corollary \ref{cor:ERK_p4N} and Theorem~\ref{thm.main}(b).}
The following two examples were used in the literature to investigate the strong stability of the explicit Runge--Kutta method with $s=p=4$.  Due to the new result in Corollary \ref{cor:ERK_p4N}, it is clear that this scheme is \emph{not} strongly stable. However, in view of Theorem \ref{thm.LS}, since $c_{p+1}=c_{p+2}=0$ and hence $\delta_{p+1}>0$, the method is strongly stable w.r.t.~$\cLSHscalar$.
We deduce from Theorem~\ref{thm.main}(a) that the method is strongly stable w.r.t.~$\cLAS^{m}$ with $m=1$.
Since~\eqref{c.even} holds, the considered Runge--Kutta scheme is \emph{not} strongly stable w.r.t.~$\{\mA\}$ if the semi-dissipative matrix $\mA\in\cLAS$ satisfies $2\mHC(\mA)+1>p$, i.e., if $\mHC(\mA)>3/2$; see Theorem~\ref{thm.main}(b).

\begin{example} \label{ex:Levy.Tadmor}
Levy and Tadmor~\cite[\S3.5]{LeTa98} use the example 
\begin{equation*} 
\mA 
=-5\begin{pmatrix} 
1 & 2 & 2 & 2 & 2 \\
0 & 1 & 2 & 2 & 2 \\
0 & 0 & 1 & 2 & 2 \\ 
0 & 0 & 0 & 1 & 2 \\
0 & 0 & 0 & 0 & 1
\end{pmatrix}
\quad\text{with }
\mAH 
=-5\begin{pmatrix} 
1 & 1 & 1 & 1 & 1 \\
1 & 1 & 1 & 1 & 1 \\
1 & 1 & 1 & 1 & 1 \\
1 & 1 & 1 & 1 & 1 \\
1 & 1 & 1 & 1 & 1 
\end{pmatrix}
\end{equation*}
to test numerically the strong stability of the explicit Runge--Kutta method with $s=p=4$.
The matrix~$\mA$ is semi-dissipative, but $\|R(\tau\mA)\| >1$ for some $\tau>0$.
The staircase form of $(\mJ,\mR)=(\mAS,\mAH)$ is 
{\small
\begin{equation*}
 \mV\mA \mV^*
=-5 \left[\begin{array}{ccccc}
5 & -2 \sqrt{2} & 0 & 0 & 0 
\\
 2 \sqrt{2} & 0 & \frac{\sqrt{35}}{5} & 0 & 0 
\\
 0 & -\frac{\sqrt{35}}{5} & 0 & -\frac{4 \sqrt{35}}{35} & 0 
\\
 0 & 0 & \frac{4 \sqrt{35}}{35} & 0 & \frac{\sqrt{7}}{7} 
\\
 0 & 0 & 0 & -\frac{\sqrt{7}}{7} & 0 
\end{array}\right] ,
\quad
 \mV
=\left[\begin{array}{ccccc}
\frac{\sqrt{5}}{5} & \frac{\sqrt{5}}{5} & \frac{\sqrt{5}}{5} & \frac{\sqrt{5}}{5} & \frac{\sqrt{5}}{5} 
\\
 \frac{\sqrt{10}}{5} & \frac{\sqrt{10}}{10} & 0 & -\frac{\sqrt{10}}{10} & -\frac{\sqrt{10}}{5} 
\\
 -\frac{\sqrt{14}}{7} & \frac{\sqrt{14}}{14} & \frac{\sqrt{14}}{7} & \frac{\sqrt{14}}{14} & -\frac{\sqrt{14}}{7} 
\\
 -\frac{\sqrt{10}}{10} & \frac{\sqrt{10}}{5} & 0 & -\frac{\sqrt{10}}{5} & \frac{\sqrt{10}}{10} 
\\
 \frac{\sqrt{70}}{70} & -\frac{2 \sqrt{70}}{35} & \frac{3 \sqrt{70}}{35} & -\frac{2 \sqrt{70}}{35} & \frac{\sqrt{70}}{70} 
\end{array}\right].
\end{equation*}}
In particular, $r=6$, $n_1=\cdots=n_5=1$, $n_6=0$.
Hence, the matrix~$\mA$ has the HC-index $\mHC=r-2=4$ (see Remark \ref{stair-HC}).
\end{example}

\begin{example} \label{ex:Sun.Shu}
To show that the explicit Runge--Kutta method with $s=p=4$ is not strongly stable, Sun and Shu~\cite[Prop.~1.1]{SunShu17} use the counter-example 
\begin{equation*} 
\mA 
=-\begin{pmatrix} 
1 & 2 & 2 \\
0 & 1 & 2 \\
0 & 0 & 1 
\end{pmatrix}
\quad\text{with }
\mAH 
=-\begin{pmatrix} 
1 & 1 & 1 \\
1 & 1 & 1 \\
1 & 1 & 1 
\end{pmatrix}.
\end{equation*}
The matrix~$\mA$ is semi-dissipative, but $\|R(\tau\mA)\| >1$ for $\tau>0$ sufficiently small.
The staircase form of $(\mJ,\mR)=(\mAS,\mAH)$ is 
\[
 \mV\mA \mV^*
=-\left[\begin{array}{ccc}
3 & \frac{2 \sqrt{3}\, \sqrt{2}}{3} & 0 
\\
 -\frac{2 \sqrt{3}\, \sqrt{2}}{3} & 0 & -\frac{\sqrt{2}\, \sqrt{6}}{6} 
\\
 0 & \frac{\sqrt{2}\, \sqrt{6}}{6} & 0 
\end{array}\right],
\quad
 \mV
=\left[\begin{array}{ccc}
\frac{\sqrt{3}}{3} & \frac{\sqrt{3}}{3} & \frac{\sqrt{3}}{3} 
\\
 -\frac{\sqrt{2}}{2} & 0 & \frac{\sqrt{2}}{2} 
\\
 -\frac{\sqrt{6}}{6} & \frac{\sqrt{6}}{3} & -\frac{\sqrt{6}}{6} 
\end{array}\right].
\]
Here, we have $r=4$, $n_1=n_2=n_3=1$, $n_4=0$, and $\mA$ has the HC-index $\mHC=r-2=2$.
Due to Theorem~\ref{thm.main}(b), matrix~$\mA$ has the minimal HC-index~$\mHC(\mA)$ to detect the failure of strong stability, such that $\mHC(\mA)>3/2$.
\end{example}


\section{Proof of Theorem \ref{thm.main}}\label{sec.main}

The proof is based on the short-time decay of the spectral norm of the matrix exponential $t\mapsto e^{t\mA}\in \Cnn$, characterized by the HC-index of $\mA$. We recall from \cite[Theorem 2.7(a)]{AAC22}:

\begin{proposition}\label{prop:ODE-short}
Let the matrix $\mA\in\Cnn$ be semi-dissipative.
Then~$\mA$ is asymptotically stable (with HC-index $\mHC\in\N_0$) if and only if
\begin{equation}\label{short-t-decay}
  \|e^{t\mA}\|_2 = 1-ct^a+\bigO(t^{a+1})\quad\text{ for } t\in[0,\epsilon), %
\end{equation}
for some $a,c,\epsilon>0$. In this case, necessarily $a=2m_{HC}+1$.
\end{proposition}

The sharp multiplicative factor~$c$ in~\eqref{short-t-decay} has been determined explicitly in~\cite[Theorem 2.7(b)]{AAC22}. 
Definition \ref{def:HCI} immediately yields the following observation.

\begin{remark} \label{rem:u0}
Let $\mA$ be semi-dissipative with HC-index $\mHC \in\N_0$. 
In view of~\eqref{mHC-ineq}, there exists a normalized vector $u_0$ such that
\begin{equation}\label{u0_prep}
 u_0 \in\ker \big(T_{\mHC-1}\big)
 =\ker \bigg(\sum_{j=0}^{\mHC-1} \mAS^j \mAH (\mAS^*)^j\bigg) ,
\end{equation}
but none that would satisfy instead also $(-\mAH)^{1/2} (\mAS)^{\mHC} u_0 =0$.
Hence, there exists a normalized vector $u_0$ such that
\begin{equation}\label{u0}
\|(-\mAH)^{1/2} (\mAS)^j u_0\|_2 =0\ \text{ for }\ 0 \leq j \leq \mHC-1,
\quad (-\mAH)^{1/2} (\mAS)^{\mHC} u_0\ne 0 . 
\end{equation}
\end{remark}


\medskip\noindent\textbf{Preparation.}
We claim that $t\mapsto \|e^{t\mA}\|_2$ and  $\tau\mapsto\|R(\tau\mA)\|_2$ are real analytic on a sufficiently small time interval. In particular, the Taylor expansions of order $p+1$ of $\|e^{t\mA}\|_2$ and $\|R(\tau\mA)\|_2$ about zero exist.
Indeed, according to \cite[Lemma 1]{Ko01},
for sufficiently small time $t_0>0$, there exists a real analytic function $\Phi: [0,t_0]\to\R$ such that $\|e^{t\mA}\|_2 =\Phi(t)$
for all $t\in[0,t_0]$. This statement can also be derived from \cite[Chap.~2, \S6]{Ka95} or \cite[Theorem 4.3.17]{HiPr05}.

Similarly, we prove that $\|R(\tau\mA)\|_2$ is (at least) locally real analytic for sufficiently small $\tau\geq 0$. 
The stability function $R(\tau\mA)$ is an analytic matrix function w.r.t.\ $\tau$, hence $G(\tau):=R(\tau\mA^*)R(\tau\mA)$ is again an analytic matrix function. 
Due to \cite[Theorem 4.3.17]{HiPr05}, there exists a neighborhood of $\tau_0=0$ and locally analytic functions $\lambda_j(\tau) =\sigma_j(G(\tau))$, where $\sigma_j$ denotes the singular values of (the self-adjoint) matrix~$G(\tau)$, and $\lambda_j(0)=1$ for $j=1,...,n$. Hence, 
\[
 \|R(\tau\mA)\|_2 = \max_{j=1,\ldots,n} 
 \Big\{ \sqrt{\lambda_j(R(\tau\mA))}\Big\} .
\]
If two real analytic functions are the same on some converging sequence then they are identical.
Since $\lambda_j(0)=1$, we use null-sequences to deduce that any two real analytic functions $\lambda_k$, $\lambda_\ell$ on $[0,\infty)$ are either identical or there exists $\tau_*^{k\ell}>0$ such that they do not intersect on a finite interval $[0,\tau_*^{k\ell})$.
Since there are only finitely many $\lambda_j$, there exists $\tau_*:=\min\{\tau_*^{k\ell}\}>0$ such that $\|R(\tau\mA)\|_2$ is real analytic on $[0,\tau_*)$, see also the discussion in~\cite[Chap.~2, \S6.4]{Ka95}.


\medskip\noindent\textbf{Proof of Theorem \ref{thm.main}, statement~(a).}
Following the proof of \cite[Theorem 2.7]{AAC22}, we consider the spectral norm $\|e^{t\mA}\|_2^2 =\|Q(t)\|_2=\lambda_{\rm max}(Q(t))$ for small $t>0$, where 
\begin{equation*}
  Q(t):=e^{t\mA^*}e^{t\mA}
  = \sum_{j=0}^{\infty} \frac{t^j}{j!}\ U_j, \quad
  U_j = (\mA+\mA^*)^j = \sum_{k=0}^j \binom{j}{k} (\mA^*)^k \mA^{j-k},
\end{equation*}
satisfying $\|U_j\|_2 \leq (2 \|\mA\|_2)^j$ for $j\in\N_0$.

First, we compute the stability function $R(\tau\mA)$ in~\eqref{1.R}, setting $c_k=0$ for $k>s$ and $c_k=1$ for $k\le p$,
\begin{align*}
 G(\tau) 
&=R(\tau\mA^*) R(\tau\mA)
=\sum_{j=0}^{2s} \frac1{j!} \sum_{k=0}^{j} \binom{j}{k} c_k c_{j-k} (\tau \mA^*)^k (\tau \mA)^{j-k}
\\
&=\sum_{j=0}^{2s} \frac{\tau^j}{j!} \sum_{k=0}^{j} \binom{j}{k} c_k c_{j-k} (\mA^*)^k \mA^{j-k}
\\
&=\sum_{j=0}^{p} \frac{\tau^j}{j!} U_j
+\sum_{j=p+1}^{2s} \frac{\tau^j}{j!} \sum_{k=0}^{j} \binom{j}{k} c_k c_{j-k} (\mA^*)^k \mA^{j-k} \\
&= \sum_{j=0}^p\frac{\tau^j}{j!} U_j
+ \frac{\tau^{p+1}}{(p+1)!}\bigg(\sum_{k=1}^p\binom{p+1}{k}(\mA^*)^k
\mA^{p+1-k} \\
&\phantom{xx}+c_{p+1}\mA^{p+1}+c_{p+1}(\mA^*)^{p+1}\bigg) + \bigO(\tau^{p+2}).
\end{align*}
We subtract
\begin{align*}
  Q(\tau) &= \sum_{j=0}^p\frac{\tau^j}{j!} U_j
  + \frac{\tau^{p+1}}{(p+1)!}\bigg(\sum_{k=1}^p\binom{p+1}{k}(\mA^*)^k
  \mA^{p+1-k} \\
  &\phantom{xX}+\mA^{p+1}+(\mA^*)^{p+1}\bigg) + \bigO(\tau^{p+2})
\end{align*}
from $G(\tau)$ and set $\widetilde{c}_k=c_k-1$ to find that
\begin{equation}\label{GminusQ}
  G(\tau)-Q(\tau) = \frac{\tau^{p+1}}{(p+1)!}\widetilde{c}_{p+1}
  \big(\mA^{p+1} +(\mA^*)^{p+1}\big) + \bigO(\tau^{p+2}).
\end{equation}
Consequently, the Taylor expansions for $G(\tau)$ and $Q(\tau)$ differ starting from the $\tau^{p+1}$-term.
Now, consider an explicit Runge--Kutta scheme of order $p\in \N$ with stability function $R(z)$, and matrices~$\mA\in\cLAS$ with HC-index $\mHC$ satisfying $2\mHC+1\leq p$.
Then, taking into account Proposition~\ref{prop:ODE-short}, there exist $c$, $\epsilon>0$ such that
\begin{align*}
  \|R(\tau\mA)\|_2 
  &= \|G(\tau)\|_2^{1/2} = \|Q(\tau)\|_2^{1/2} + \bigO(\tau^{p+1}) \\
  & =\|e^{\tau\mA}\|_2 +\bigO(\tau^{p+1})
  = 1-c\tau^{2\mHC+1}+\bigO(\tau^{2\mHC+2})
\end{align*}
for all $\tau\in[0,\epsilon)$.
We conclude that $\|R(\tau\mA)\|_2\le1$ for sufficiently small $\tau$, hence proving that all explicit Runge--Kutta schemes of order $p\in \N$ are strongly stable w.r.t.~matrices~$\mA\in\cLAS$ with HC-index $\mHC$ satisfying $2\mHC+1\leq p$.


\medskip{\noindent\textbf{Proof of Theorem \ref{thm.main}, statement~(b) for even order $p$.}
Let $p\in 2\N$ and let $\mA\in\cLAS$ with HC-index $\mHC=p/2$. 
The proof uses the following result, which follows from 
\cite[Lemma A.4]{AAC22}, when using the simple observation from \cite[Lemma 2.4]{AAC22} that $\ker(T_m)=\ker[\sum_{j=0}^m (\mA^*)^j \mAH\mA^j]$.

\begin{lemma}\label{lm:g}
Let $\mA\in\Cnn$ be semi-dissipative with HC-index $\mHC\in\N$.
Then, for each $u_0\in\ker(T_{\mHC-1})$, there exists a polynomial vector function $u_\tau\in\Cn$, $\tau\in[0,1]$, of the form 
\begin{equation} \label{u_ast}
  u_\tau =u_0 +\sum_{\ell=1}^{\mHC} b_\ell \tau^\ell (-\mA)^\ell u_0
\end{equation}
for suitable numbers $b_\ell\in\R$, such that for $\tau\in[0,1]$,
\begin{equation} \label{SolutionNorm:expansion}
 g(u_\tau;\tau) 
:= u_\tau^* \bigg( \sum_{j=1}^{\infty} \frac{\tau^j}{j!} U_j \bigg) u_\tau 
 = -2 c_1(u_0) \tau^a +\bigO(\tau^{a+1}), 
\end{equation}
where $a = 2\mHC+1$ and $c_1(u_0) := \|(-\mAH)^{1/2} \mA^\mHC u_0\|_2^2/((2\mHC+1)! \binom{2\mHC}{\mHC})$.
\end{lemma}

The vector function $u_\tau$ of this lemma is used as an approximation of an eigenfunction of $G(\tau)=R(\tau\mA)^*R(\tau\mA)$, whose associated eigenvalue is larger than one for sufficiently small $\tau>0$.
Since $g(u_\tau;\tau)=u_\tau^* Q(\tau)u_\tau-\|u_\tau\|_2^2$, we can write
\begin{equation} \label{cLAS:p_even:b1}
 u_\tau^* G(\tau) u_\tau
  =\|u_\tau\|_2^2 +g(u_\tau;\tau) +u_\tau^* 
  \big(G(\tau) -Q(\tau)\big)u_\tau, 
\end{equation}
where here and in the following, the function $g$ is considered on $\C^n\times[0,\infty)$. As $\mA$ is semi-dissipative, Remark \ref{rem:u0} yields the existence of some $u_0\in\ker(T_{\mHC-1})\setminus\{0\}$ satisfying \eqref{u0}.
This shows that
\begin{align}
 u_0^* \big(\mA^{p+1} +(\mA^*)^{p+1} \big) u_0
&=u_0^* \big((\mAH +\mAS)^{p+1} +(\mAH -\mAS)^{p+1} \big) u_0 \nonumber \\
&=2 (-1)^{\mHC}\ u_0^* (\mAS^*)^{\mHC} \mAH (\mAS)^{\mHC} u_0 \nonumber \\
&=-2 (-1)^{\mHC}\|(-\mAH)^{1/2} \mAS^{\mHC} u_0\|_2^2 \ne 0, \label{cLAS:p_even:b2}
\end{align}
where we used from~\eqref{u0} that $\mAH (\mAS)^j u_0 =0$ and $u_0^* (\mAS^*)^j \mAH=0$ for $0 \leq j \leq \mHC-1$.
Therefore, we infer from \eqref{GminusQ} and \eqref{u_ast} that
\begin{align*}
  u_\tau^*\big(G(\tau)-Q(\tau)\big)u_\tau &= -2(-1)^{m_{HC}}
  \frac{\tau^{2m_{HC}+1}}{(2m_{HC}+1)!}\widetilde{c}_{p+1}
  \|(-\mAH)^{1/2} \mAS^{\mHC} u_0\|_2^2 \\
  &\phantom{xX}+\bigO(\tau^{2\mHC+2}).
\end{align*}
Inserting this expression and \eqref{SolutionNorm:expansion} into \eqref{cLAS:p_even:b1} yields
\begin{align}
  u_\tau^*& G(\tau) u_\tau
  = \|u_\tau\|_2^2 +g(u_\tau;\tau) + u_\tau^* \big(G(\tau)-Q(\tau)\big) u_\tau \nonumber \\
  &= \|u_\tau\|_2^2 -\frac{2\tau^{2\mHC+1}}{(2\mHC+1)!} \Big( \tbinom{2\mHC}{\mHC}^{-1} +(-1)^{\mHC} \tc_{p+1}\Big) 
  \|(-\mAH)^{1/2} \mAS^{\mHC} u_0\|_2^2 \nonumber \\
  &\phantom{xX}+\bigO(\tau^{2\mHC+2}). \label{uGu:p_even}
\end{align}
By condition \eqref{c.even}, the expression in the round brackets is negative,
$$
   \tbinom{2\mHC}{\mHC}^{-1} +(-1)^{\mHC} \tc_{p+1} 
  = \tbinom{p}{p/2}^{-1} +(-1)^{p/2} (c_{p+1} -1) < 0,
$$
implying that $u_\tau^*G(\tau)u_\tau > \|u_\tau\|_2^2$ for all $\tau\in(0,\tau_*)$ for some $\tau_*>0$. We conclude that $\|R(\tau)\|_2>1$ on that interval. Thus, the Runge--Kutta scheme with even order $p$ is not strongly stable under condition \eqref{c.even}.


\medskip\noindent\textbf{Proof of Theorem \ref{thm.main}, statement (b) for odd order $p$.}
Let $p\in~2\N_0+1$ and let $\mA\in\cLAS$ with HC-index $\mHC=(p+1)/2$.
As before, Remark \ref{rem:u0} yields the existence of some $u_0\in\ker (T_{\mHC-1})\setminus\{0\}$ satisfying \eqref{u0} and hence, proceeding similarly as in~\eqref{cLAS:p_even:b2},
\begin{align*}
  u_0^* \big(\mA^{p+1} +(\mA^*)^{p+1} \big)u_0
  &= u_0^*\big((\mAH +\mAS)^{2\mHC} +(\mAH -\mAS)^{2\mHC}\big)u_0 \\
  &= 2 u_0^* \mAS^{2\mHC} u_0
  =2 (-1)^{\mHC} \|(\mAS)^{\mHC} u_0\|_2^2 \neq 0.
\end{align*}
Then, with the vector function $u_\tau$ as in Lemma \ref{lm:g} and taking into account \eqref{GminusQ} and $2m_{HC}=p+1$, 
\begin{align}
  u_\tau^* G(\tau) u_\tau
  &= \|u_\tau\|_2^2 +g(u_\tau;\tau) 
  + u_\tau^* \big(G(\tau) -Q(\tau)\big) u_\tau \nonumber \\
  &= \|u_\tau\|_2^2 -2c_1(u_0) \tau^{2\mHC+1} +\bigO(\tau^{2\mHC+2}) \nonumber \\
  &\phantom{xX}+ \frac{\tau^{p+1}}{(p+1)!} \tc_{p+1} u_0^*
  \big(\mA^{p+1} +(\mA^*)^{p+1} \big) u_0 +\bigO(\tau^{p+2}) \nonumber \\
  &= \|u_\tau\|_2^2 +\frac{\tau^{p+1}}{(p+1)!} \tc_{p+1} u_0^*\big(\mA^{p+1} +(\mA^*)^{p+1} \big)u_0 +\bigO(\tau^{p+2}) \nonumber \\
  &= \|u_\tau\|_2^2 +\frac{2\tau^{p+1}}{(p+1)!}(-1)^{(p+1)/2}\tc_{p+1} \|\mAS^{(p+1)/2} u_0\|_2^2 +\bigO(\tau^{p+2}). \label{uGu:p_odd}
\end{align}
Since $\|\mAS^{(p+1)/2} u_0\|_2^2\ne 0$, condition \eqref{c.odd}
implies that 
$$
  (-1)^{(p+1)/2}\tc_{p+1}=(-1)^{(p+1)/2}(c_{p+1}-1)>0
$$ 
and consequently, $u_\tau^* G(\tau) u_\tau>\|u_\tau\|_2^2$ for all $\tau\in(0,\tau_*)$ for some $\tau_*>0$. 
Thus, the Runge--Kutta scheme with odd order $p$ is not strongly stable under condition \eqref{c.odd}.

\medskip\noindent\textbf{Proof of Theorem \ref{thm.main}, statement~(c)}}.
Consider an explicit Runge--Kutta scheme of order $p\in\N$ which is not strongly stable and let \eqref{c} hold.
If a matrix~$\mA\in\cLAS$ with HC-index $\mHC=\mHC(\mA)$ satisfies $2\mHC+1> p$, then we can repeat the construction in the proof of statement~(b):
Consider $m\in\N$ such that $2m+1>p$ but $2m-1\leq p$.
Due to $\mHC\geq m$, there exists a vector $u_0\in\Cn$ such that
\begin{equation*}
 \|(-\mAH)^{1/2} \mAS^j u_0\|_2 =0 \text{ for } 0\leq j\leq m-1, \quad
 \|(-\mAH)^{1/2} \mAS^m u_0\|_2 \ne 0,
\end{equation*}
see~\eqref{u0}, but with $m$ replacing $\mHC$. 
Using a vector function $u_\tau$ for $\tau\in[0,1]$ as constructed in~\eqref{u_ast}, we can use the above equations~\eqref{uGu:p_even}, \eqref{uGu:p_odd}, again with $m$ replacing $\mHC$, to prove statement~(c) in Theorem~\ref{thm.main}. 
\hfill $\Box$


\section{Proof of Theorem \ref{thm.StrongStabilityInWeakForm}}
\label{sec:StrongStability.of.ERK:weak.form}

As pointed out in~\cite[\S3.5]{LeTa98}, an explicit Runge--Kutta method of order $p=4$ with $s=4$ stages can be made strongly stable w.r.t.~$\cLAS$ when restricting to some adapted $\mP$-norm.
Motivated by this observation, we prove Theorem~\ref{thm.StrongStabilityInWeakForm}.

First, we observe that 
for strong stability in weak form, the analogous statement as in Theorem~\ref{thm.stab} holds, i.e., an explicit Runge--Kutta method is strongly stable in weak form if and only if it is strongly stable in weak form w.r.t.~$\cLAS$ and~$\cLSHscalar$.

Second, let an asymptotically stable, semi-dissipative matrix~$\mA\in\cLAS$ be given.
A matrix~$\mA\in\Cnn$ is asymptotically stable if and only if there exists a matrix $\mP\in\PDHn$ such that  
\begin{equation} \label{strict_ineq:Lyapunov}
 \mA^* \mP +\mP\mA < 0 ,
\end{equation}  
see e.g.~\cite[Corollary 15.10.1]{Be18}, \cite[Theorem 3.17]{HadChe08}. 
Following the proof of Lemma \ref{lem.sim}, we find that 
$\hA:=\mP^{1/2}\mA\mP^{-1/2}$ is dissipative. Hence, $\hA$ has the HC-index $\mHC(\hA)=0$.
Consider the solution $(u^k)_{k\in\N_0}$ in~\eqref{1.RK} in the weighted norm~$\|\cdot\|_\mP$ with $\mP\in\PDHn$ satisfying~\eqref{strict_ineq:Lyapunov}. By Step 1 of the proof of Theorem \ref{thm.stab}, strong stability in weak form for the Runge--Kutta scheme, i.e. $\|u^1\|_\mP\leq \|u^0\|_\mP$, is equivalent to 
\begin{equation}\label{u-tilde-decay}
  \|\widehat{u}^1\|\leq \|\widehat{u}^0\|,\quad \mbox{where } \widehat{u}^1=R(\tau\hA)\widehat{u}^0.
\end{equation}
The inequality in \eqref{u-tilde-decay} indeed holds due to Theorem~\ref{thm.main}(a), since $\hA$ has the HC-index $\mHC=0$.
Thus, $\|u^1\|_\mP\leq \|u^0\|_\mP$ holds without (further) constraints on $\mA$.

Finally, consider strong stability w.r.t.~$\cLSHscalar$, the set of purely imaginary scalars $\lambda\in \I\R$. 
Introducing a ``weighted norm''~$\|\cdot\|_\mP$ for some positive constant~$\mP\in\R^+$ does not change the stability criterion $\|u^1\|_\mP\leq \|u^0\|_\mP$ on the considered explicit Runge--Kutta scheme.
This finishes the proof.
\hfill $\Box$

\begin{example}[Continuation of Example~\ref{ex:Sun.Shu}]
While the explicit Runge-Kutta scheme with $s = p = 4$ steps is \emph{not} strongly stable (due to Corollary \ref{cor:ERK_p4N}), it is strongly stable in weak form (due to Theorem \ref{thm.StrongStabilityInWeakForm}, since it is strongly stable w.r.t.~$\cLSHscalar$; see Table~\ref{table:ERK}).  
Notice that the authors of \cite{SunShu17} study this problem in the Euclidean norm, whereas we use an adapted $\mP$-norm to guarantee the decay estimate $\|u^1\|_\mP\leq \|u^0\|_\mP$.

In fact, the Jordan normal form $-\mA =(\mW^*)^{-1} \mJ \mW^*$ of $-\mA$ holds with
\[
\mJ 
=\begin{pmatrix}
1 & 1 & 0 \\
0 & 1 & 1 \\
0 & 0 & 1 
\end{pmatrix}
\quad\text{and}\quad
\mW 
=\begin{pmatrix}
1 & 0 & 0 \\
0 & 2 & 0 \\
0 & 2 & 4 
\end{pmatrix}.
\]
Moreover, a suitable matrix $\mP=\mP(\mA)$ is given as $\mP=\mW\mW^*>0$. 
Then, the explicit Runge-Kutta scheme with $s=p=4$ stages is strongly stable w.r.t.~$\{\mA\}$ and this specific $\mP$-norm. 
\end{example}


\textbf{Acknowledgement:}
The authors were supported by the Austrian Science Fund (FWF) via the FWF-funded SFB \# F65. The third author acknowledges support from the FWF, \# P33010. 
This work received funding from the European Research Council (ERC) under the European Union's Horizon 2020 research and innovation programme, ERC Advanced Grant NEUROMORPH, no.~101018153.



\end{document}